\documentclass[a4paper,11pt]{amsart} 
\usepackage{graphicx}
\usepackage{mathrsfs}

\usepackage[ansinew]{inputenc}    

\usepackage{color}
\usepackage{amssymb,amscd}

\usepackage[all,cmtip]{xy}

\usepackage{tikz-cd}

\newtheorem{thm}{Theorem}[section]
\newtheorem{cor}[thm]{Corollary}
\newtheorem{lem}[thm]{Lemma}
\newtheorem{prop}[thm]{Proposition}
\newtheorem{conj}[thm]{Conjecture}

\theoremstyle{definition}
\newtheorem{defn}[thm]{Definition}

\theoremstyle{definition}
\newtheorem{rem}[thm]{Remark}

\theoremstyle{definition}

\def\C{\mathbb C}
\def\R{\mathbb R}

\def\km{\mathfrak m}
\def\cC{\mathscr C}
\def\cO{\mathcal O}

\def\dim{\operatorname{dim}}

\def\codim{\operatorname{codim}}

\def\id{\operatorname{id}}

\def\im{\operatorname{Im}}

\def\supp{\operatorname{supp}}
\def\coker{\operatorname{coker}}
\def\Tor{\operatorname{Tor}}
\begin{document}

\newcommand{\correct}[1]{{\color{red}{#1}}}
\newcommand{\red}[1]{{\color{red}{#1}}}

\author[J. F. Bobadilla, J.J. Nuño, G. Peñafort]{J. Fernández de Bobadilla, J. J.~Nu\~no-Ballesteros, G.~Peñafort-Sanchis}

\title[A Jacobian module for disentanglements]{A Jacobian module for disentanglements and applications to Mond's conjecture}

\address{(1) IKERBASQUE, Basque Foundation for Science, Maria Diaz de Haro 3, 48013, Bilbao, Spain
(2) BCAM, Basque Center for Applied Mathematics, Mazarredo 14, E48009 Bilbao, Spain}

\email{jbobadilla@bcamath.org}

\address{Departament de Geometria i Topologia,
Universitat de Val\`encia, Campus de Burjassot, 46100 Burjassot
SPAIN}

\email{Juan.Nuno@uv.es}

\address{Instituto de Matemática Pura e Aplicada, Estrada Dona Castorina 110, Rio de Janeiro, 22460-320, Brazil}

\email{penafort@impa.br}

\thanks{First author is supported by ERCEA 615655 NMST Consolidator Grant, MINECO by the project reference 
MTM2013-45710-C2-2-P, by the Basque Government through the BERC 2014-2017 program, 
by Spanish Ministry of Economy and Competitiveness MINECO: BCAM Severo Ochoa excellence accreditation SEV-2013-0323 and 
by Bolsa Pesquisador Visitante Especial (PVE) - Ciências sem Fronteiras/CNPq Project number:  401947/2013-0. 
The second and third authors are partially supported by DGICYT Grant MTM2015--64013--P. The third author is supported by Bolsa P\'os-Doutorado J\'unior - CNPq Project number 401947/2013-0.}

\subjclass[2000]{Primary 58K15; Secondary 58K40, 58K65} \keywords{Image Milnor number, $\mathscr A_e$-codimension, weighted homogeneous}

\begin{abstract} 
Given a germ $f:(\C^n,S)\to (\C^{n+1},0)$, we define an $\mathcal O_{n+1}$-module $M(f)$ with the property that  $\mathscr A_e$-$\codim(f)\le \dim_\C M(f)$, with equality if 
$f$ is weighted homogeneous. 
We also define a relative version $M_y(F)$ for unfoldings $F=(u,f_u)$, in such a way that $M_y(F)$ specialises to $M(f)$ when $u=0$. The main result is that if $(n,n+1)$ are 
nice dimensions, then $\dim_\C M(f)\ge \mu_I(f)$, with equality if and only if $M_y(F)$ is Cohen-Macaulay, for some stable unfolding $F$. Here, $\mu_I(f)$ denotes the image
Milnor number of $f$, so that if $M_y(F)$ is Cohen-Macaulay, then we have Mond's conjecture for $f$; furthermore, if $f$ is quasi-homogeneous Mond's conjecture for $f$ is
equivalent to the fact that $M_y(F)$ is Cohen-Macaulay. Finally, we observe that to prove Mond's conjecture, it is enough to 
prove it in a suitable family of examples.
\end{abstract}

\maketitle

\section{Introduction}

For any hypersurface with isolated singularity $(X,0)$, we have
$\tau(X,0)\le \mu(X,0)$, 
with equality if $(X,0)$ is weighted homogeneous. Here, $\tau(X,0)$ is the Tjurina number, that is, the minimal number of parameters in a versal deformation of $(X,0)$ and $\mu(X,0)$ is the Milnor number,
which is the number of spheres in the Milnor fibre of $(X,0)$. If $g\colon (\C^{n+1},0)\to(\C,0)$ is a function such that $g=0$ is a reduced equation of $(X,0)$, then we can compute both numbers in terms of $g$:
$$
\tau(X,0)=\dim_\C\frac{\mathcal O_{n+1}}{J(g)+\langle g\rangle},\quad \mu(X,0)=\dim_\C\frac{\mathcal O_{n+1}}{J(g)},
$$
where $\mathcal O_{n+1}$ is the local ring of holomorphic germs from $(\C^{n+1},0)$ to $\C$ and $J(g)$ denotes the Jacobian ideal generated by the partial derivatives of $g$. Thus, the initial statement about $\tau$ and $\mu$ becomes evident.
The Jacobian algebra deforms flatly over the parameter space of any deformation $g_t$ of $g$, it is known to encode crucial properties of the vanishing cohomology and its monodromy by its relation with the Brieskorn lattice and it is crucial
in the construction of Frobenius manifold structures in the bases of versal unfoldings. See the works of Brieskorn, Varchenko, Steenbrink, Scherk, Hertling and others, and the  
the books~\cite{AGV},~\cite{Kulk} and~\cite{Hert} 

Inspired by the previous inequality, D. Mond \cite{Mond:vanishing}  tried to obtain a result of the same nature in the context of 
singularities of mappings. He considered a hypersurface $(X,0)$ given by the image of a map germ 
$f\colon (\C^n,S)\to (\C^{n+1},0)$, with $S\subset \C^n$ a finite set and which has isolated instability under the action of the Mather group $\mathscr A$ of biholomorphisms in the source and the target. The Tjurina number has to be substituted by the $\mathscr A_e$-codimension, which is equal to the minimal number of parameters in an $\mathscr A$-versal deformation of $f$. Instead of the Milnor fibre, one considers the disentanglement, that is, the image $X_u$ of a stabilisation $f_u$ of $f$. Then, $X_u$ has the homotopy type of a wedge of spheres and Mond defined the image Milnor number $\mu_I(f)$ as the number of such spheres. Note that, outside the range of Mather's dimensions, some germs do not admit a stabilisation. Then, he stated the following conjecture:

\begin{conj} Let $f\colon (\C^n,S)\to(\C^{n+1},0)$ be an $\mathscr A$-finite map germ, with $(n,n+1)$ nice dimensions. Then,
$$
\mathscr A_e\text{-}\codim(f)\le \mu_I(f),
$$
with equality if $f$ is weighted homogeneous.
\end{conj}

The conjecture is known to be true for $n=1,2$ (see \cite{JS,Mond:vanishing,Mond:bent wires}) but it remains open until now for $n\ge 3$. There is a related result for map germs $f\colon (\C^n,S)\to(\C^p,0)$ with $n\ge p$, where one considers $\Delta$ the discriminant of $f$ instead of its image and defines the discriminant Milnor number $\mu_\Delta(f)$ in the same way. Damon and Mond showed in \cite{DM} that if $(n,p)$ are nice dimensions, then $\mathscr A_e$-$\codim(f)\le \mu_\Delta(f)$ with equality if $f$ is weighted homogeneous. There are many papers in the literature with related results, partial proofs and examples in which the conjecture has been checked. We refer to \cite{Mond:Edinburgh} for a recent account of these results. 

Going back to hypersurface singularities $g$, but now with non-isolated singularities, it is not anymore clear the relation of the Jacobian algebra of $g$  with the vanishing cohomology. 
Moreover it is apparent in easy examples that the Jacobian algebra does not
deform flatly in unfoldings. In fact the possibility of studying the vanishing cohomology via deformations that simplify the critical set (in the same vein that Morsifications do for isolated singularities) does not exist in general. However,
for restricted classes of singularities Siersma, Pellikaan, Zaharia, Nemethi, Marco-Buzun\'ariz and the first author have developed methods that allow to split the vanishing cohomology of a non-isolated singularity in two direct summands
according with the geometric properties of a deformation $g_u$ of $g$ which plays the role of a Morsification (one may find a nice survey in~\cite{Siersma3}).
The first is a free vector space contributing to the middle dimension cohomology of the Milnor fibre, with as much dimension as the number of Morse points that appear away from the zero set of $g_u$ ($u\neq 0$), 
the second is determined by the non-isolated singularities of the zero-set of $g_u$ ($u\neq 0$).

Given $f\colon (\C^n,S)\to(\C^{n+1},0)$, an $\mathscr A$-finite map germ, we consider a generic $1$-parameter deformation $f_u$ of it (a stabilisation). Let $g_u$ be the equation defining the image of $f_u$. It turns out that the deformation 
$g_u$ is suitable to split the vanishing cohomology of $g$ in two direct summands, as explained in the paragraph above, and that the first summand corresponds with the cohomology of the image $X_u$, whose rank is the image Milnor number.
The main novelty of this paper is the definition of an Artinian $\cO_{n+1}$ -module $M(f)$, which satisfies
$$\dim_\C M(f)=\mathscr A_e\text{-}\codim(f)+\dim_\C((g)+J(g)/J(g))$$
and, in the nice dimensions, this dimension upper bounds the image Milnor number. Moreover we define a relative version $M_y(F)$ of the module for unfoldings $F$ of $f$, 
and we prove that when we specialise the parameter the relative module specialises to the original $M(f)$. 

The first main result of this paper is Theorem~\ref{main1}, which implies that the dimension of $M(f)$ equals the image Milnor number 
if and only if $M_y(F)$ is flat over the base of the unfolding. We also prove that this is equivalent to the flatness of the Jacobian algebra over the base of the unfolding.
Thus, under the flatness condition, $M(f)$ is expected to play the role of the Milnor algebra for isolated singularities, in the sense of encoding the first direct summand of the vanishing
cohomology, which is the only one present for isolated singularities. It is very interesting to investigate whether the relation
of the vanishing cohomology of isolated singularities with the Jacobian algebra explained admit a generalisation to a relation between the first direct summand of the vanishing cohomology of $g$ and the module $M(f)$.

The second main result (Theorem~\ref{main2}) says that the flatness of $M_y(F)$ implies Mond's conjecture for $f$, and it is equivalent to it if $f$ is weighted homogeneous. In Theorems~\ref{reduction1}~and~\ref{reduction2} we derive the surprising
consequence that in order to settle Mond's conjecture in complete generality it is enough to prove it for a series of examples of increasing multiplicity.

In Section~\ref{algoritmo} we derive formulas to compute the module the modules $M(f)$ and $M_y(F)$ which are well suited for computer algebra programs and also lead to new formulas for the $\mathscr A_e$-codimension: see Corollary~\ref{nf0} and 
Remark~\ref{nf}.

The authors would like to thank Duco van Straten, David Mond and Craig Huneke for useful conversations.

\section{The $\mathscr A_e$-codimension and the image Milnor number}\label{section:image}

We recall the definition of codimension of a map germ with respect to the Mather $\mathscr A$ group. Let $f\colon (\C^n,S)\to (\C^p,0)$ be any holomorphic map {multi-}germ.
We denote by $\cO_n=\mathcal O_{\C^n,S}$ and $\cO_p=\mathcal O_{\C^p,0}$ the rings of holomorphic function germs in the source and the target respectively.
We also denote by $\theta_n=\theta_{\C^n,S}$ and $\theta_p=\theta_{\C^p,0}$ the corresponding modules of germs of vector fields and by $\theta(f)$ the module of germs of vector fields along $f$. 
Then, we have two associated morphisms: $tf\colon \theta_{n}\to\theta(f)$ given by $tf(\eta)=df\circ\eta$ and $\omega f\colon \theta_{p}\to\theta(f)$ given by $\omega(\xi)=\xi\circ f$. The $\mathscr A_e$-codimension of $f$ is defined as:
$$
\mathscr A_e\text{-}\codim(f)=\dim_\C \frac{\theta(f)}{tf(\theta_{n})+\omega f(\theta_p)}.
$$
We say that $f$ is $\mathscr A$-stable (resp. $\mathscr A$-finite) if the $\mathscr A_e$-codimension is zero (resp. finite).

By an $r$-parameter unfolding of a map {multi-}germ $f\colon (\C^n,S)\to(\C^p,0)$ we mean another {multi-}germ
$F\colon (\C^r\times\C^n,\{0\}\times S)\to(\C^r\times\C^p,0)$ given by $F(u,x)=(u,f_u(x))$ and such that $f_0=f$.
It was proved by Mather \cite{MatherII} that $f$ is $\mathscr A$-stable if and only if any unfolding $F$ of $f$ is trivial. This means that there exist $\Phi$ and $\Psi$ unfoldings of the identity in $(\C^r\times\C^n,\{0\}\times S)$ and $(\C^r\times\C^p,0)$, respectively, such that $\Psi\circ F\circ\Phi^{-1}$ is the constant unfolding $\id\times f$.

We present now a result due to Mond which gives a way to compute the $\mathscr A_e$-codimension in the case $p=n+1$ in terms of a defining equation of the image of $f$. We need to introduce some notation. 

From now on, we will assume that $f\colon (\C^n,S)\to (\C^{n+1},0)$ is finite and generically one-to-one and denote by
$(X,0)\subset(\C^{n+1},0)$ its image. The restriction $\bar{f}\colon (\C^n,S)\to(X,0)$ is the normalization map, so the
induced morphism 
$\bar{f}^*\colon \mathcal O_{X,0}\to \mathcal O_n$ is a monomorphism and we will consider $\mathcal O_{X,0}$ as a subring
of $\mathcal O_n$. Thus, we have a commutative diagram:
$$
\begin{tikzcd}
\mathcal O_{n+1} \arrow{r}{f^*} \arrow[twoheadrightarrow]{dr}{\pi}
& \cO_n\\
& \mathcal O_{X,0} \arrow[hookrightarrow]{u}{i}
\end{tikzcd}
$$
where $\pi$ is the epimorphism induced by the inclusion $(X,0)\subset(\C^{n+1},0)$. Here, we consider both $\mathcal O_{X,0}$ and $\mathcal O_n$ as $\mathcal O_{n+1}$-modules via the corresponding morphisms. Finally, let $g\in \mathcal O_{n+1}$ be such that $g=0$ is a reduced equation of $(X,0)$ and denote by $J(g)\subset\mathcal O_{n+1}$ the Jacobian ideal of $g$.

\begin{lem}\label{IsoMondACodim}\cite[Proposition 2.1]{Mond:vanishing} 
Let $f\colon (\C^n,S)\to (\C^{n+1},0)$ be $\mathscr A$-finite, with $n\ge 2$. Then,
$$
\mathscr A_e\text{-}\codim(f)=\dim_\C\frac{J(g)\cdot \mathcal O_n}{J(g)\cdot \mathcal O_{X,0}}.
$$
\end{lem}

Note that the proof of this lemma given in \cite{Mond:vanishing} is only for monogerms, but a careful revision of the proof shows that it also works for multigerms. Note also that the lemma is not true for $n=1$. In fact, in that case (see \cite{Mond:bent wires}):
$$
\mathscr A_e\text{-}\codim(f)=\dim_\C\frac{J(g)\cdot \mathcal O_{1}}{J(g)\cdot \mathcal O_{X,0}}+\dim_C \frac{\mathcal O_{1}}{\langle f'_1,f'_2\rangle}.
$$

Next, we recall the definition of image Milnor number. Consider any $r$-parameter unfolding $F(u,x)=(u,f_u(x))$ of $f\colon (\C^n,S)\to(\C^{n+1},0)$. Denote by $(\mathcal X,0)$  the image hypersurface of $F$ in $(\C^r\times \C^{n+1},0)$, and by $X_u$ the fibre of $\mathcal X$ over $u\in\C^r$. We fix a small enough
representative  
$$F\colon W\to T\times B_\epsilon,$$
where $W,T,B_\epsilon$ are open neighbourhoods of $S$ and the origin in $\C^{r+n},\C^r,\C^{n+1}$, respectively, such that:
\begin{enumerate}
\item $F$ is finite (i.e., closed and finite-to-one), 
\item $F^{-1}(0)=\{0\}\times S$,
\item $B_\epsilon$ is a Milnor ball for the hypersurface $X_0\subset\C^{n+1}$
\item $T$ is small enough so that the intersection $X_u\cap\partial B_\epsilon$ of the hypersurface with the Milnor sphere is topologically trivial over all $u\in T$.
\end{enumerate}

In order to understand the topology of $X_u\cap B_\epsilon$ we use the following 
general result due to Siersma:

\begin{thm} \cite{Siersma}\label{Siersma}
Let $g\colon (\C^{n+1},0)\to(\C,0)$ define a reduced hypersurface $(X_0,0)$, not necessarily with isolated singularity, 
and let $G\colon (\C^r\times\C^{n+1},0)\to(\C,0)$ be a deformation of $g$ such that for all $u$, 
\begin{enumerate}
\item $\{g_u=0\}$ is topologically trivial over the Milnor sphere $\partial B_\epsilon$, and
\item all the critical points of $g_u$ which are not in $X_u=g_u^{-1}(0)\cap B_\epsilon$ are isolated.
\end{enumerate}
Then, $X_u\cap B_\epsilon$ is homotopy equivalent to a wedge of $n$-spheres and the number of such $n$-spheres is equal to
$$
\sum_{y\in B_\epsilon\setminus X_u}\mu(g_u;y),
$$
{where $\mu(g_u;y)$ denotes the Milnor number of the function $g_u$ at the point $y$.}
\end{thm}

\begin{defn}
Assume $r=1$. Given a representative $F\colon W\to T\times B_\epsilon$ as above, we say that $F$ is a \emph{stabilisation} if for any 
 $u\in T\setminus\{0\}$ and any point $y\in X_u\cap B_\epsilon$ the multigerm of $f_u$ at $y$ is $\mathscr A$-stable.
\end{defn}

It is well known that every map $f$ admits a stabilisation if $(n,n+1)$ are nice dimensions in the sense of Mather \cite{MatherVI}. As an application of Siersma's previous result, Mond proves the following theorem in \cite{Mond:vanishing}. Again the proof of the theorem given in \cite{Mond:vanishing} is only for monogerms, but it is easy to check that the proof also works for multigerms.

\begin{thm}\cite{Mond:vanishing} \label{wedge}
Let $f\colon (\C^n,S)\to(\C^{n+1},0)$ be $\mathscr A$-finite with $(n,n+1)$ nice dimensions and let $F$ be a stabilisation of $f$. Then, for any  $u\in T\setminus\{0\}$, the image of $f_u$ has the homotopy type of a wedge of $n$-spheres.  Moreover, the number of such $n$-spheres is independent of the parameter $u$ and on the stabilisation $F$.
\end{thm}

\begin{rem}
 As it is mentioned in \cite{Mond:vanishing}, if $(n,n+1)$ are not nice dimensions, the theorem is still true if we substitute a stabilisation by a $C^0$-stabilisation. This means that for any $u\in T\setminus\{0\}$ and any point $y\in X_u\cap B_\epsilon$, the multigerm of $f_u$ at $y$ is $C^0$-$\mathscr A$-stable (that is, its jet extension is multitransverse to the canonical stratification of Mather of the jet space). Since a $C^0$-stabilisation of $f$ always exists (by the Thom-Mather transversality theorem), the image Milnor number can be defined in general by taking a $C^0$-stabilisation instead of a stabilisation.
\end{rem}

\begin{defn} Let $f\colon (\C^n,S)\to(\C^{n+1},0)$ be $\mathscr A$-finite. The image $X_u$ of a $C^0$-stabilisation $f_u$ of $f$ is called the \emph{disentanglement} and number of $n$-spheres in $X_u$ is called the \emph{image Milnor number} of $f$ and is denoted by $\mu_I(f)$.
\end{defn}

\begin{rem}
Sometimes, it is more interesting to work with a stable unfolding of $f$ instead of a stabilisation. This means an $r$-parameter unfolding $F$ which is $\mathscr A$-stable as a map germ. If $f$ is finite, then it has finite singularity type and hence, it always admits a stable unfolding (see \cite{MatherIII}). 

Then, there exists a proper closed analytic subset $\mathcal B\subset T$ such that for any $u\in T\setminus\mathcal B$ and any point $y\in X_u\cap B_\epsilon$ the multigerm of $f_u$ at $y$ is $C^0$-$\mathscr A$-stable (or $\mathscr A$-stable if $(n,n+1)$ are nice dimensions). This subset $\mathcal B$ is known as the \emph{bifurcation set} of $F$. It follows that for any $u\in T\setminus\mathcal B$, the image of $f_u$ has the homotopy type of a wedge of $n$-spheres and the number of such $n$-spheres is equal to $\mu_I(f)$.

In fact, we only have to take a curve $C\subset T$ joining $u$ to the origin and such that $C\cap \mathcal B=\{0\}$. If $C$ is parametrised by $\gamma(t)$, then $H(t,x)=(t,f_{\gamma(t)}(x))$ defines a $C^0$-stabilisation of $f$ such that $f_{\gamma(t)}=f_u$ for some $t\ne0$.
\end{rem}

\section{The module $M(f)$}


We denote by $C(f)$ the conductor ideal of $\mathcal O_{X,0}$ in $\mathcal O_n$ and by $\mathscr C(f)$ its inverse image through $\pi\colon \cO_{n+1}\to \cO_{X,0}$, that is,
$$
C(f):=\{h\in \mathcal O_{X,0}:\ h\cdot\mathcal O_n\subset\mathcal O_{X,0}\},
\quad \mathscr C(f):=\pi^{-1}(C(f)).
$$
The conductor has the property that it is the largest ideal of $\mathcal O_{X,0}$ which is also an ideal of $\mathcal O_n$. We can compute easily $C(f)$ by using the following result of Piene \cite{RP} (see also Bruce-Marar \cite{BM}).

\begin{lem}\label{Piene} There exists a unique $\lambda\in O_n$ such that
$$
\frac{\partial g}{\partial y_i}\circ f=(-1)^i\lambda \det(df_1,\dots,df_{i-1},df_{i+1},\dots,df_{n+1}),\ 1\le i\le n+1,
$$
and moreover, $C(f)$ is generated by $\lambda$.
\end{lem}

From Lemma \ref{Piene} we have the inclusion $J(g)\cdot \mathcal O_n\subset C(f)$, which motivates the following definition.

\begin{defn} We define $M(f)$ as the kernel of the epimorphism of $\mathcal O_{n+1}$-modules induced by $\pi$:
$$ \frac{\mathscr C(f)}{J(g)}\longrightarrow \frac{C(f)}{J(g)\cdot \mathcal O_n}.$$
\end{defn}


\begin{prop}\label{snake}
We have the following exact sequence of $\mathcal O_{n+1}$-modules:
$$
0\longrightarrow K(g)\longrightarrow M(f)\longrightarrow \frac{J(g)\cdot \mathcal O_n}{J(g)\cdot \mathcal O_{X,0}}\longrightarrow 0
$$
where $K(g):=(\langle g\rangle+J(g))/J(g)$.
\end{prop}

\begin{proof}
Consider the following commutative diagram:
$$
\begin{CD}
@. 0  @. 0 @. 0 @.\\
@. @VVV @VVV @VVV @.\\
0@>>> K(g) @>>> M(f) @>>> \frac{J(g)\cdot \mathcal O_n}{J(g)\cdot \mathcal O_{X,0}}@>>> 0\\
@. @VV\mu_1 V @VV\mu_2 V @VV\mu_3 V @.\\
0@>>> K(g) @>>> \frac{\mathcal C(f)}{J(g)} @>>> \frac{C(f)}{J(g)\cdot \mathcal O_{X,0}}@>>> 0\\
@. @VV\lambda_1 V @VV\lambda_2 V @VV\lambda_3 V @.\\
@. 0 @>>> \frac{C(f)}{J(g)\cdot \mathcal O_n} @>>> \frac{C(f)}{J(g)\cdot \mathcal O_n}@>>> 0\\
@. @. @VVV @VVV @.\\
@. @. 0 @. 0 @.
\end{CD}
$$
Observe that all columns and the second and third rows are exact. Therefore, from the Snake Lemma we obtain an exact sequence
$$
0\to \ker(\lambda_1)\to\ker(\lambda_2)\to\ker(\lambda_3)\to\coker(\lambda_1)\to\cdots
$$
But $\coker(\lambda_1)=0$, $\ker(\lambda_1)=K(g)$, $\ker(\lambda_2)=\im(\mu_2)=M(f)$ and $\ker(\lambda_3)=\im(\mu_3)=(J(g)\cdot \mathcal O_n)/(J(g)\cdot \mathcal O_{X,0})$, so we get the desired exact sequence.
\end{proof}

\begin{cor} \label{M(f)=0}
Let $f\colon (\C^n,S)\to(\C^{n+1},0)$ be $\mathscr A$-finite with $n\ge 2$. Then:
\begin{enumerate}
\item $M(f)=0$ if and only if $f$ is $\mathscr A$-stable and $g\in J(g)$. 
\item If $\dim_\C M(f)<\infty$, then
$$
\dim_\C M(f)=\mathscr A_e\text{-}\codim(f)+\dim_\C K(g).
$$
\end{enumerate}
\end{cor}

This corollary is important, because it gives a simple method to compute the $\mathscr A_e$-codimension of a map germ $f\colon (\C^n,S)\to(\C^{n+1},0)$, with $n\ge 2$, just by means of a reduced equation of the image. We will explain this with more details in Section \ref{algoritmo}.

\begin{rem}\label{nice}
If $(n,n+1)$ are nice dimensions, then the condition that $\dim_\C M(f)<\infty$ in part (2) of Corollary \ref{M(f)=0} is not necessary since it is a consequence of the $\mathscr A$-finiteness of $f$. In fact, all stable singularities in the nice dimensions are $\mathscr A$-equivalent to weighted homogeneous singularities (see \cite{MatherVI}), hence the module $K(g)$ is supported only at the origin. The same applies to $\frac{J(g)\cdot \mathcal O_n}{J(g)\cdot \mathcal O_{X,0}}$, because of Lemma~\ref{IsoMondACodim}. Now, the claim follows immediately from the exact sequence in Proposition~\ref{snake}.
\end{rem}

We finish this section with a couple of interesting properties about the ideals $C(f)$ and $\mathscr C(f)$ which will be used later.

\begin{rem}\label{rem:determinantal} It follows from the proof of \cite[Theorem 3.4]{MP} that $\mathscr C(f)$ coincides with the first Fitting ideal of $\mathcal O_n$ as an $\mathcal O_{n+1}$-module via 
$f^*\colon \mathcal O_{n+1}\to \mathcal O_n$ (that is, $\mathscr C(f)$ is the ideal generated by the submaximal  minors of a matrix presentation of $\mathcal O_n$). Furthermore, the same theorem also states that $\mathcal O_{n+1}/\mathscr C(f)$ is a determinantal ring of dimension $n-1$. 
By \cite[Proposition 1.5]{MP}, the zero locus of $\mathscr C(f)$ is
$$V(\mathscr C(f))=\left\{y\in\C^{n+1}:\ \sum_{f(x)=y}\dim_\C\frac{\mathcal O_{\C^n,x}}{f^*\mathfrak m_{\C^{n+1},y}}>1 \right\},
$$
which is equal to the points $y\in\C^{n+1}$ such that either $y=f(x)$ and $x$ is a non-immersive point of $f$ or $y=f(x)=f(x')$ with $x\ne x'$. Hence, we deduce that $V(\mathscr C(f))$ is the singular locus of $(X,0)$. This space is also known as the target double point space of $f$. 
\end{rem}

\begin{rem}\label{rem:ramif} 
Another consequence of Lemma \ref{Piene} is that multiplicacion by $\lambda$ induces an isomorphism:
$$
\frac{C(f)}{J(g)\cdot \mathcal O_n}\cong \frac{\mathcal O_n}{R(f)},
$$
where $R(f)\subset \mathcal O_n$ is the ramification ideal, that is, the ideal generated by the maximal minors of the Jacobian matrix of $f$. If $f$ is $\mathscr A$-finite and $n\ge 2$, then $\mathcal O_n/R(f)$ is a determinantal ring (of dimension $n-2$ in this case). The zero locus $V(R(f))\subset (\C^n,S)$ is the set of non-immersive points of $f$.
\end{rem}

						\section{The relative version for unfoldings}
We are interested in the behavior of the module $M(f)$ under deformations. With this motivation, we define a relative version of this module for unfoldings. Let $F$ be an $r$-parameter unfolding of $f\colon (\C^n,S)\to(\C^{n+1},0)$. We have a commutative diagram:
$$
\begin{CD}
(\C^r\times \C^n,\{0\}\times S) @>F>> (\C^r\times\C^{n+1},0)\\
@AiAA @AAjA\\
(\C^n,S) @>f>> (\C^{n+1},0),
\end{CD}
$$
where $i(x)=(0,x)$ and $j(y)=(0,y)$. This induces another commutative diagram:
\begin{equation}\label{CD}
\begin{CD}
\mathcal O_{r+n+1} @>F^*>> \mathcal O_{r+n}\\
@Vj^*VV @VVi^*V\\
\mathcal O_{n+1} @>f^*>> \mathcal O_{n},\\
\end{CD}
\end{equation}
whose columns are epimorphisms. The conductor ideal $C(f)$ and its inverse image $\mathscr C(f)$ behave well under deformations, meaning that
$$
i^*(C(F))=C(f),\quad j^*(\mathscr C(F))=\mathscr C(f).
$$
The claim for $C(f)$ follows immediately from Piene's Lemma \ref{Piene} and the claim for $\mathscr C(f)$ is a consequence of the first one and the commutative diagram (\ref{CD}):
\begin{align*}
j^*(\mathscr C(F))&=j^*\left((F^*)^{-1}(C(F))\right)=(f^*)^{-1}\left(i^*(C(F))\right)\\
&=(f^*)^{-1}(C(f))=\mathscr C(f).
\end{align*}

Now we need an ideal which gives a deformation of the Jacobian ideal $J(g)$. Let $G\in\mathcal O_{r+n+1}$ be such that $G=0$ is a reduced equation of $(\mathcal X,0)$ of $F$ and such that $j^*(G)=g$. It is not true that $j^*(J(G))=j(g)$ since $J(G)$ contains the additional partial derivatives with respect to the parameters $u_i$. Instead of it, we consider the relative Jacobian ideal $J_{y}(G)$, that is, the ideal in $\mathcal O_{r+n+1}$ generated by the partial derivatives of $G$ with respect to the variables $y_i$, $i=1,\dots,n+1$. Then, we get:
$$
j^*(J_y(G))=J(g).
$$

\begin{defn}
We define $M_{y}(F)$ as the kernel of the epimorphism of $\mathcal O_{r+n+1}$-modules induced by $F^*$:
$$
 \frac{\mathscr C(F)}{J_{y}(G)}\longrightarrow \frac{C(F)}{J_{y}(G)\cdot \cO_{r+n}}.
$$
\end{defn}

The main result of this section will be that the module $M_y(F)$ specialises to $M(f)$ when $u=0$, that is,
$$
M_y(F)\otimes\dfrac{\cO_r}{\km_r}\cong M(f),
$$
where $\km_r$ is the maximal ideal of $\mathcal O_{r}$, and the isomorphism is induced by the epimorphism $j^*$. Although the ideals $\mathscr C(F), C(F), J_y(G)$ specialise to $\mathscr C(f), C(f), J(g)$ as ideals, respectively, it is not so obvious from the definition that $M_y(F)$ specialises to $M(f)$.

From now on in this section and unless otherwise stated, the symbol $\cong$ will be used to represent an isomorphism of modules induced by $j^*$.

	\begin{lem}\label{tensorsCCaligraphic}
For any $r$-parameter unfolding $F$ of $f$, we have:
$$
\cC(F)\otimes\dfrac{\cO_r}{\km_r}\cong \cC(f),
$$
and moreover $I\cdot\cC(F)=I\cap\cC(F)$, where $I=\km_r\cdot O_{r+n+1}$.
\end{lem}

\begin{proof}
Since $I$ is the kernel of $j^*$, we have
{$(\mathscr C(F)+I/I)=  j^*(\mathscr C(F))=\mathscr C(f)$} and from this we deduce:
$$
\frac{\cO_{r+n+1}}{\cC(F)}\otimes\dfrac{\cO_r}{\km_r}= \frac{\cO_{r+n+1}}{\cC(F)+I}
=  \frac{\cO_{r+n+1}/I}{(\cC(F)+I)/I}\cong \dfrac{\cO_{n+1}}{\cC(f)}.
$$
Take the exact sequence of $\cO_r$-modules
$$0\longrightarrow\cC(F)\longrightarrow\cO_{r+n+1} \longrightarrow \frac{\cO_{r+n+1}}{\cC(F)}\longrightarrow 0$$
and consider the induced long exact Tor-sequence:
$$\dots\longrightarrow\Tor_1^{\cO_r}\left(\frac{\cO_{r+n+1}}{\cC(F)},\dfrac{\cO_r}{\km_r}\right)\longrightarrow\cC(F)\otimes\dfrac{\cO_r}{\km_r}\longrightarrow\cO_{n+1}\longrightarrow \dfrac{\cO_{n+1}}{\cC(f)}\longrightarrow 0.$$

By Remark \ref{rem:determinantal}, $\cO_{r+n+1}/\cC(F)$ is determinantal of dimension $r+n-1$. Then, it is Cohen-Macaulay and since the fibre $\cO_{n+1}/\cC(f)$ has dimension $n-1$, it is $\cO_r$-flat. Therefore,
$$
\Tor_1^{\cO_r}\left(\frac{\cO_{r+n+1}}{\cC(F)},\dfrac{\cO_r}{\km_r}\right)=0,
$$
and the above exact sequence implies
$$
\cC(F)\otimes\dfrac{\cO_r}{\km_r}\cong \cC(f).
$$

To show the second part, on one hand we have  
$$\cC(f)\cong  \dfrac{\cC(F)+I}{I}= \dfrac{\cC(F)}{\cC(F)\cap I}.$$
On the other hand, we have  
$$\cC(F)\otimes\dfrac{\cO_r}{\km_r}=  \dfrac{\cC(F)}{\km_r\cdot\cC(F)}=\dfrac{\cC(F)}{I\cdot\cC(F)},$$
and the result follows from the first part of the lemma.
\end{proof}

An analogous proof shows the following:
	
	\begin{lem}\label{lem:C(F)}
	For any $r$-parameter unfolding $F$ of $f$, we have:
$$
C(F)\otimes\dfrac{\cO_r}{\km_r}\cong  C(f),
$$
and moreover $L\cdot C(F)=L\cap C(F)$, where $L=\km_r\cdot \cO_{r+n}$.
%
\end{lem}

%
%

	\begin{prop}\label{prop:special}
For any $r$-parameter unfolding $F$ of $f$, the following hold:
\begin{enumerate}
\item$\displaystyle\dfrac{\mathscr C(F)}{J_y(G)}\otimes\dfrac{\cO_{r}}{\km_r}\cong \dfrac{\mathscr C(f)}{J(g)}$,
\item$\displaystyle\dfrac{C(F)}{J_y(G)\cdot \cO_{n+r}}\otimes\dfrac{\cO_{r}}{\km_r}\cong \dfrac{C(f)}{J(g)\cdot \cO_{n}}$.
\end{enumerate}

\begin{proof}
We show item (1), since the proof of item (2) is analogous. In order to simplify the notation, we write $\mathscr C:=\mathscr C(F)$ and $J:=J_y(G)$. By Lemma  \ref{tensorsCCaligraphic},
\begin{align*}
\dfrac{\mathscr C}{J}\otimes\dfrac{\cO_{r}}{\km_r}
&\cong \dfrac{\mathscr C /J }{I\cdot(\mathscr C /J )}
=  \dfrac{\mathscr C /J }{(I\cdot\mathscr C +J )/J }\\
&= \dfrac{\mathscr C }{I\cap\mathscr C +J }
= \dfrac{\mathscr C /I\cap\mathscr C }{(I\cap\mathscr C +J )/I\cap\mathscr C }\\
&= \dfrac{\mathscr C /I\cap\mathscr C }{J /I\cap\mathscr C \cap J }
= \dfrac{\mathscr C /I\cap\mathscr C }{J /I\cap J }\\
&= \dfrac{(\mathscr C +I)/I}{(J +I)/I}\cong \dfrac{\cC(f)}{J(g)}.
\end{align*}
\end{proof}
\end{prop}

Next lemma shows that over the source ring $\mathcal O_{r+n}$, the Jacobian ideals $J(G)$ and $J_y(G)$ coincide.

	\begin{lem}\label{lem:jacob-source}
	For any $r$-parameter unfolding $F$ of $f$, we have:
$$J_y(G)\cdot \mathcal O_{r+n}=J(G)\cdot \mathcal O_{r+n}.$$
\end{lem}
\begin{proof}
Let us write $F(u,x)=(u,f_u(x))$, then the Jacobian matrix of $F$ has the following format:
$$
dF=
\left(
\begin{array}{cc}
I_r  &  0  \\
*  &  df_u   \\
\end{array}
\right),
$$
where $df_u$ is the Jacobian matrix of $f_u$, but considered with entries in $\mathcal O_{r+n}$. 
Denote by $M_1,\dots,M_r,M'_1,\dots,M'_{n+1}$ the {$r+n$}-minors of $dF$ in such a way that $M'_1,\dots,M'_{n+1}$ are the {$n$}-minors of $df_u$. Then $M_1,\dots,M_r$ can be generated from the other minors $M'_1,\dots,M'_{n+1}$. That is, we can put
$$
M_i=\sum_j a_{ij} M_j',
$$
for some $a_{ij}\in\mathcal O_{r+n}$. Now, by Piene's Lemma \ref{Piene}:
$$
\frac{\partial G}{\partial u_i}\circ F=\Lambda M_i,\quad \frac{\partial G}{\partial y_j}\circ F=\Lambda M'_j,
$$
where $\Lambda$ is the generator of the conductor ideal $C(F)$. We have:
$$
\frac{\partial G}{\partial u_i}\circ F=\sum a_{ij} \frac{\partial G}{\partial y_j}\circ F.
$$
\end{proof}

Now we arrive to the main result of this section.

	\begin{thm}\label{thmfamilia}
If $F$ is any $r$-parameter unfolding of $f$, then:
$$M_y(F)\otimes \dfrac{\cO_{r}}{\km_r}\cong M(f).$$

\begin{proof}
We have a short exact sequence coming from the definition of $M_y(F)$:
$$
0\longrightarrow M_y(F)\longrightarrow \frac{\mathscr C(F)}{J_y(G)}\longrightarrow
\frac{C(F)}{J_y(G)\cdot\mathcal O_{r+n}}\longrightarrow 0.
$$
By tensoring with $\cO_r/\km_r$ and by the results of Proposition \ref{prop:special}, we obtain the following associated long exact Tor-sequence:
\begin{align*}
\dots&\longrightarrow\Tor_1^{\cO_r}\left(\frac{C(F)}{J_y(G)\cdot\mathcal O_{r+n}},\dfrac{\cO_r}{\km_r}\right)\\ 
&\quad\longrightarrow M_y(F)\otimes \dfrac{\cO_{r}}{\km_r}\longrightarrow\frac{\mathscr C(f)}{J(g)}\longrightarrow\frac{C(f)}{J(g)\cdot\mathcal O_{n}}\longrightarrow 0.
\end{align*}

We claim that $C(F)/J_y(G)\cdot\mathcal O_{r+n}$ is $\cO_r$-flat. In fact, by Lemma \ref{lem:jacob-source}, 
$$
\frac{C(F)}{J_y(G)\cdot\mathcal O_{r+n}}=\frac{C(F)}{J(G)\cdot\mathcal O_{r+n}}\cong\frac{\cO_{r+n}}{R(F)},$$
where $R(F)$ is the ramification ideal and the isomorphism here is induced by multiplication of the generator of $C(F)$ (see Remark \ref{rem:ramif}). But $\cO_{r+n}/R(F)$ is determinantal of dimension $r+n-2$. Then, it is Cohen-Macaulay and since the fibre $\cO_{n}/R(f)$ has dimension $n-2$, it is $\cO_r$-flat. Therefore, 
$$
\Tor_1^{\cO_r}\left(\frac{C(F)}{J_y(G)\cdot\mathcal O_{r+n}},\dfrac{\cO_r}{\km_r}\right)=0,
$$
and from the above exact sequence we get:
$$M_y(F)\otimes \dfrac{\cO_{r}}{\km_r}\cong M(f).$$
\end{proof}
\end{thm}

We finish this section with the next proposition, which gives the relative version of the short exact sequence of Proposition \ref{snake} for unfoldings. The proof is based on a commtutative diagrama analogous to that appearing in the proof of \ref{snake}, but with the relative version of the modules.

\begin{prop}\label{snake-family}
Let $F$ be an $r$-parameter unfolding of $f$. We have the following exact sequence of $\mathcal O_{r+n+1}$-modules:
$$
0\longrightarrow K_y(G)\longrightarrow M_y(F)\longrightarrow \frac{J_y(G)\cdot \mathcal O_{r+n}}{J_y(G)\cdot \mathcal O_{\mathcal X,0}}\longrightarrow 0
$$
where $K_y(G):=(\langle G\rangle+J_y(G))/J_y(G)$.
\end{prop}

\begin{rem}
By using analogous arguments to those of Theorem~\ref{thmfamilia}, it is not difficult to prove that the module on the right hand side of the above exact sequence specialises to the module which controls the $\mathscr A_e$-codimension when $n\ge 2$. More precisely, we have
$$
\frac{J_y(G)\cdot \mathcal O_{r+n}}{J_y(G)\cdot \mathcal O_{\mathcal X,0}}\otimes
\dfrac{\cO_{r}}{\km_r}\cong
\frac{J(g)\cdot \mathcal O_{n}}{J(g)\cdot \mathcal O_{X,0}}.
$$
The proof follows easily by using the short exact sequence:
$$
0\longrightarrow \frac{J_y(G)\cdot \mathcal O_{r+n}}{J_y(G)\cdot \mathcal O_{\mathcal X,0}}\longrightarrow \frac{C(F)}{J_y(G)\cdot \mathcal O_{\mathcal X,0}}\longrightarrow \frac{C(F)}{J_y(G)\cdot \mathcal O_{r+n}}\longrightarrow 0
$$
After tensoring with $\mathcal O_r/\km_r$ and taking into account that the module on the right hand side is $\mathcal O_r$-flat, we get the desired result.

It is not true in general that the module $K_y(G)$ specialises to $K(g)$. That is, we may have that 
$$K_y(G)\otimes\frac{\cO_r}{\km_r}\not\cong K(g).
$$
In fact, by using the short exact sequence of Proposition \ref{snake-family}, if $K_y(G)\otimes\frac{\cO_r}{\km_r}\cong K(g)$, this would imply that $\frac{J_y(G)\cdot \mathcal O_{r+n}}{J_y(G)\cdot \mathcal O_{\mathcal X,0}}$ is $\mathcal O_r$-flat. But it is obvious that this module is not flat when $f$ is $\mathscr A$-finite and $F$ is a stabilisation of $f$, since it is supported only at the origin.
\end{rem}

\section{An equivalent description of the module $M(f)$}\label{algoritmo}

In this section we show a description of the modules $M(f)$ and $M_y(F)$ which is better suited for applications. Proposition \ref{M(f)} allows us to compute $M(f)$ easily using a computer algebra system, such as \textsc{Singular} \cite{singular}. Since $K(g)$ can be computed as well, from Corollary~\ref{M(f)=0} we obtain an expression for the $\mathscr A_e$-codimension of any map germ $f\colon (\C^n,S)\to(\C^{n+1},0)$, with $n\ge 2$.

\begin{prop}\label{M(f)}
Let $f\colon (\C^n,S)\to(\C^{n+1},0)$ be any map germ and $F$ any $r$-parameter unfolding of $f$. Then:
\begin{align*}
M(f)&=\dfrac{(f^*)^{-1}(J(g)\cdot \mathcal O_{n})}{J(g)},\\
M_y(F)&=\dfrac{(F^*)^{-1}(J(G)\cdot \mathcal O_{r+n})}{J_y(G)}. 
\end{align*}
\end{prop}

\begin{proof} By construction, $M(f)$ is given by
$$
M(f)=\frac{(f^*)^{-1}(J(g)\cdot \mathcal O_n)\cap \mathscr C(f)}{J(g)}.
$$
But Lemma \ref{Piene} implies the inclusion $J(g)\cdot \mathcal O_n\subset C(f)$, hence we have the inclusion 
$$(f^*)^{-1}(J(g)\cdot \mathcal O_n)\subset (f^*)^{-1}(C(f))=\mathscr C(f).$$
The proof for $M_y(F)$ is analogous, {taking into account that $J_y(G)\cdot\mathcal O_{r+n}$ equals $J(G)\cdot\mathcal O_{r+n}$ by Lemma~\ref{lem:jacob-source}.}
\end{proof}

\begin{cor}\label{corFomulaM_y(F)} Let $F$ be a stable unfolding of $f$. Then, 
$$M_y(F)=\frac{J(G)+\langle G\rangle}{J_y(G)}.$$
\end{cor}

\begin{proof}
Since $F$ is stable, $M(F)=K(G)$ by {Proposition~\ref{snake}}. Now, the first part Proposition~\ref{M(f)} implies that $(F^*)^{-1}(J(G)\cdot \mathcal O_{r+n}))=J(G)+\langle G\rangle$ and the result follows from 
the second part of Proposition~\ref{M(f)}.
\end{proof}

\begin{defn} Let $F$ be an unfolding of $f$. We say that $G$ is a \emph{good defining equation} for $F$ if $G=0$ is a 
reduced equation of the image of $F$ and moreover $G\in J(G)$.
\end{defn}

Note that there always exists a stable unfolding $F$ which admits a good defining equation. In fact, if $F(u,x)=(u,f_u(x))$ is any $r$-parameter stable unfolding, then we take $F'$ as the 1-parameter trivial unfolding of $F$, that is, $F'(t,u,x)=(t,u,f_u(x))$. Let $G=0$ be a reduced equation of the image of $F$ and take $G'(t,u,y)=e^t G(u,y)$. Then $G'=0$ is a reduced equation of the image of $F'$ and $\partial G'/\partial t=G'$, hence $G'\in J(G')$.

\begin{cor} Let $F$ be a stable unfolding of $f$ and $G$ a good defining equation for $F$. Then, 
$$M_y(F)=\frac{J(G)}{J_y(G)}.$$
\end{cor}

\begin{cor}\label{nf0}Let $f\colon (\C^n,S)\to(\C^{n+1},0)$ be $\mathscr A$-finite, with $n\ge 2$ and $(n,n+1)$ nice dimensions. Let $F$ be a stable unfolding of $f$, then
$$\mathscr A_e\text{-}\codim(f)=\dim_{\C}\left(\frac{J(G)+\langle G\rangle}{J_y(G)}\otimes\frac{\cO_r}{\km_r}\right)-\dim_\C\left(\frac{J(g)+\langle g\rangle}{J( g)}\right).$$

\begin{proof} It follows immediately by putting together Corollary \ref{nice}, Theorem~\ref{thmfamilia}, 
and Corollary \ref{corFomulaM_y(F)}.
\end{proof}

	\end{cor}

	\begin{rem}\label{nf}
Observe that if $G$ is a good defining equation for $F$, then we have
 $$\mathscr A_e\text{-}\codim(f)=\dim_{\C}\left(\frac{J(G)}{J_y(G)}\otimes\frac{\cO_r}{\km_r}\right)-\dim_\C\left(\frac{J(g)+\langle g\rangle}{J( g)}\right).$$
 If, moreover, $f$ is  weighted homogeneous, then 
 $$\mathscr A_e\text{-}\codim(f)=\dim_{\C}\left(\frac{J(G)}{J_y(G)}\otimes\frac{\cO_r}{\km_r}\right).$$
 \end{rem}

\section{Flatness and the Cohen-Macaulay property}

In this section, we assume that $f:(\C^n,S)\to (\C^{n+1},0)$ is a germ such that $\dim_\C M(f)<\infty$
and $F$ is an $r$-parameter unfolding of $f$.  By Theorem~\ref{thmfamilia} and the Preparation Theorem, $M_y(F)$ is finite over $\mathcal O_r$. We consider a small enough representative $F\colon W\to T\times B_\epsilon$ with the properties required in Section~2 and such that the restriction of the projection onto the parameter space
$$
\pi:\supp \mathcal M_y(F)\to T
$$ is finite and $\pi^{-1}(0)=\{0\}$. Here $\mathcal M_y(F)$ is the coherent sheaf of modules on $T\times B_\epsilon$ whose stalk at the origin is $M_y(F)$. We also denote by $M_y(F)_{(u,p)}$ the stalk of $\mathcal M_y(F)$ at $(u,p)\in T\times B_\epsilon$. We have the following standard fact from commutative algebra and analytic geometry:



\begin{lem}\label{stdcomm}
The following assertions are equivalent:
\begin{enumerate}
 \item the module $M_y(F)$ is flat over $\mathcal O_r$,
 \item the module $M_y(F)$ is free over $\mathcal O_r$,
 \item the number 
 $$\Theta(u):=\sum_{p\in B_\epsilon} \dim_\C\left(M_y(F)_{(u,p)}\otimes\frac{\cO_{T,u}}{\km_{T,u}}\right)$$ 
 is independent of $u\in T$, where $\km_{T,u}$ denotes the maximal ideal 
of $\mathcal O_{T,u}$,
 \item the module $M_y(F)$ is Cohen-Macaulay of dimension $r$.
\end{enumerate}

\end{lem}


Recall that we have denoted by $\mathcal X$ the image of $F$ and by $X_u$ the fibre of {$\mathcal X$} over $u\in T$. Given a point $p\in X_u\cap B_\epsilon$, in the next theorem we denote by $M(f_u)_p$ the module $M$ computed for the germ of $f_u$
at the point $p$.

\begin{thm}\label{Cohen-Macaulay} Let $F$ be an unfolding of a map germ 
$f\colon (\C^n,S)\to(\C^{n+1},0)$, with $\dim_\C M(f)<\infty$. We have the inequality:
$$
\dim_\C M(f)\ge\sum_{p\in X_u\cap B_\epsilon} \dim_\C M(f_u)_p+b_n(X_u\cap B_\epsilon),
$$
where $b_n(X_u\cap B_\epsilon)$ denotes the $n$-th Betti number of $X_u\cap B_\epsilon$.
Moreover, the equality holds if and only if the module $M_y(F)$ is Cohen-Macaulay of dimension $r$(equivalently, if it is $\mathcal O_r$-flat).
\end{thm}

\begin{proof}
By the upper semicontinuity of $\Theta$ we have 
$$\Theta(0)\geq \Theta(u)$$
for any $u\in T$, with equality if and only if $M_y(F)$ is Cohen-Macaulay of dimension $r$, by
Lemma~\ref{stdcomm}. Let us identify both sides of the inequality.

The left hand side is equal to $\dim_\C M_y(F)\otimes(\cO_r/\km_r)=\dim_\C M(f)$ by Theorem~\ref{thmfamilia}. By the same reason, if $p\in X_u\cap B_\epsilon$, then 
$$M_y(F)_{(u,p)}\otimes\frac{\cO_{T,u}}{\km_{T,u}}=M(f_u)_p.$$ 
We split the sum $\Theta(u)$ as
$$\Theta(u)=\sum_{p\in X_u\cap B_\epsilon} \dim_\C M(f_u)_p+\sum_{p\in B_\epsilon\setminus X_u} \dim_\C \left(M_y(F)_{(u,p)}\otimes\frac{\cO_{T,u}}{\km_{T,u}}\right).$$

The first summand coincides with the first summand of the right hand side of the desired inequality. For the second summand we use the short exact sequence of Proposition \ref{snake-family}. If $p\in B_\epsilon\setminus X_u$, then
$$M_y(F)_{(u,p)}\otimes\frac{\cO_{T,u}}{\km_{T,u}}=K_y(G)_{(u,p)}\otimes\frac{\cO_{T,u}}{\km_{T,u}}=
\frac{\mathcal O_{T\times B_\epsilon,(u,p)}}{J_y(G)}\otimes\frac{\cO_{T,u}}{\km_{T,u}}=
\frac{\mathcal O_{B_\epsilon,p}}{J(g_u)},$$
which is the Jacobian algebra of $g_u$ at $p$. Thus, the second summand is equal to 
$$\sum_{p\in B_\epsilon\setminus X_u} \mu(g_u;p),$$
where $\mu(g_u;p)$ is the Milnor number of $g_u$ at $p$. By Siersma's Theorem \ref{Siersma}, 
the sum of all Milnor numbers $\mu(g_u;p)$, with $p\notin X_u$, is equal to the Betti number $b_n(X_u\cap B_\epsilon)$.
\end{proof}

The above theorem has two interesting particular cases, namely, when the unfolding $F$ is either a $C^0$-stabilisation or a stable unfolding.

\begin{cor}\label{Cohen-Macaulay-stb}
Let $F$ be un unfolding of a map germ 
$f\colon (\C^n,S)\to(\C^{n+1},0)$, with $\dim_\C M(f)<\infty$. Assume $F$ is either stable or a $C^0$-stabilisation and let $\mathcal B\subset T$ be the bifurcation set. 
For any $u\in T\setminus \mathcal B$ the following inequality holds
$$\dim_\C M(f)\ge\sum_{p\in X_u\cap B_\epsilon} \dim_\C M(f_u)_p+\mu_I(f),$$
with equality if and only if $M_y(F)$ is Cohen-Macaulay of dimension $r$ (equivalently $\mathcal O_r$-flat). {In the nice dimensions we obtain 
$$\dim_\C M(f)\ge\mu_I(f),$$
with equality if and only if $M_y(F)$ is Cohen-Macaulay of dimension $r$.}
\end{cor}

\begin{rem}
Since $\km_{r}$ is an ideal of parameters of $M_y(F)$, if we denote by $e(\km_{r};M_y(F))$ the multiplicity of $M_y(F)$
with respect to $\km_{r}$ we have 
$$
\dim_\C M(f)\ge e(\km_{r};M_y(F)),
$$ 
with equality if and only if $M_y(F)$ is Cohen-Macaulay of dimension $r$. 

The function $\Theta$ is Zariski upper-semicontinuous over $T$.
By genericity of flatness we know that there exists a proper closed subset $\Delta\subset T$ such that  $M_y(F)$ is flat over $T\setminus\Delta$ and such that the $\Theta$ jumps up 
precisely at $u\in \Delta$.

For all $u\notin \Delta$, by conservation of multiplicity, we have the equalities
\begin{align*}
e(\km_{r};M_y(F))&=\sum_{p\in B_\epsilon} e(\mathfrak m_{T,u}; M_y(F)_{(u,p)})\\
         &=\sum_{p\in B_\epsilon} \dim_\C \left(M_y(F)_{(u,p)}\otimes\frac{\cO_{T,u}}{\km_{T,u}}\right)\\
         &=\sum_{p\in X_u\cap B_\epsilon} \dim_\C M(f_u)_p+\sum_{p\in B_\epsilon\setminus X_u}\mu(g_u;p)\\
         &=\sum_{p\in X_u\cap B_\epsilon} \dim_\C M(f_u)_p+b_n(X_u\cap B_\epsilon).
\end{align*}
If, furthermore $u\notin \mathcal B\cup\Delta$, we have the equality 
$$e(\km_{r};M_y(F))=\sum_{p\in X_u\cap B_\epsilon} \dim_\C M(f_u)_p+\mu_I(f).$$

This last equality can be rewritten as 
$$
e(\km_{r};M_y(F))-\mu_I(f)=\sum_{p\in X_u\cap B_\epsilon} \dim_\C M(f_u)_p.
$$
Thus, the sum on the right hand side is independent of the generic parameter $u$. It can be proved that it is also 
independent of the stable unfolding. In fact, if we have another stable unfolding $F'$ with parameters $v_1,\dots,v_s$, 
then we can take $F''$ as the sum of the two unfoldings $F,F'$. Since $F'$ is stable, $F''$ is trivial and we can assume 
that $F''$ is constant on the parameters $v_1,\dots,v_r$. Then, $v_1,\dots,v_r$ are a regular sequence for $M_y(F'')$, so
\begin{align*}
e(\langle u_1,\dots,u_r,v_1,\dots,v_s\rangle;M_y(F''))&=e(\langle u_1,\dots,u_r\rangle;\frac{M_y(F'')}{\langle v_1,\dots,v_s\rangle M_y(F'')})\\&=e(\langle u_1,\dots,u_r\rangle;M_y(F)).
\end{align*}
We define:
$$
\alpha(f):= \sum_{p\in X_u\cap B_\epsilon} \dim_\C M(f_u)_p,
$$
where $F$ is any stable unfolding of $f$ and $u\notin \mathcal B\cup\Delta$.

Note that $\alpha(f)=0$ when $(n,n+1)$ are nice dimensions, since in that case all the singularities of $f_u$ are $\mathscr A$-stable and are $\mathscr A$-equivalent to a weighted homogeneous map germ. 
\end{rem}
%
%
%

We can prove now two of our main results:

\begin{thm}\label{main1}
Let $f\colon (\C^n,S)\to(\C^{n+1},0)$ be a germ with $\dim_\C M(f)<\infty$. 
The following statements are equivalent:
\begin{enumerate}
\item $\dim_\C M(f)=\alpha(f)+\mu_I(f)$,
\item $M_y(F)$ is Cohen-Macaulay of dimension $r$ for some stable unfolding,
\item $M_y(F)$ is Cohen-Macaulay of dimension $r$ for any stable unfolding,
\item $M_y(F)$ is Cohen-Macaulay of dimension $1$ for any $C^0$-stabilisation,
\item $\cO_{{r+n+1}}/J_y(G)$ is $\mathcal O_r$-flat for some stable unfolding,
\item $\cO_{{r+n+1}}/J_y(G)$ is $\mathcal O_r$-flat for any stable unfolding,
\item $\cO_{{n+2}}/J_y(G)$ is $\mathcal O_1$-flat for any $C^0$-stabilisation.
\end{enumerate}
Suppose that $(n,n+1)$ are nice dimensions, then the following two further statements are equivalent to the previous
ones:
\begin{enumerate}
 \item[(8)] $M_y(F)$ is Cohen-Macaulay of dimension $1$ for some stabilisation,
 \item[(9)] $\cO_{{n+2}}/J_y(G)$ is $\mathcal O_1$-flat for some stabilisation.
\end{enumerate}
\end{thm}
\begin{proof}
The equivalences between (1), (2) and (3) follow immediately from the previous results. Let us see (3) $\Rightarrow$ (4). Let $F$ be any $C^0$-stabilisation of $f$ given by $F=(t,f_t)$. We take $F'=(u,f'_u)$ a stable unfolding and consider $F''=(t,u,f''_{t,u})$ as the sum of the two unfoldings $F,F'$. Since $F'$ is stable, $F''$ is also stable and $M_y(F'')$ is Cohen-Macaulay of dimension $r+1$ by hypothesis. For any $t\ne0$, the germ of $f_t=f''_{t,0}$ is $C^0$-stable at any $p\in X_t$, hence $(t,0)\notin\mathcal B$, the bifurcation set of $F''$. By Corollary \ref{Cohen-Macaulay-stb},
$$
\dim_\C M(f)=\sum_{p\in X_t\cap B_\epsilon} \dim_\C M(f_t)_p+\mu_I(f),
$$
therefore $M_y(F)$ is also Cohen-Macaulay of dimension $1$.

We prove now that (4) $\Rightarrow$ (1). Let $F$ be any stable unfolding given by $F=(u,f_u)$. 
Choose an analytic path germ $\gamma:(\C,0)\to (\C^r,0)$ such that $\gamma(t)\notin\mathcal B\cup\Delta$, for any $t\ne0$. The unfolding $\gamma^*F$ induced by $\gamma$ is a $C^0$-stabilisation, hence $M_y(\gamma^*F)$ is Cohen-Macaulay of dimension $1$. By Corollary \ref{Cohen-Macaulay-stb}, for any $u=\gamma(t)$ with $t\ne0$,
$$
\dim_\C M(f)=\sum_{p\in X_{u}\cap B_\epsilon} \dim_\C M(f_{u})_p+\mu_I(f)=\alpha(f)+\mu_I(f).
$$

The implication (4) $\Rightarrow$ (8) is trivial and if $(n,n+1)$ are nice dimensions, then $\alpha(f)=0$, hence we also have (8) $\Rightarrow$ (1), by Corollary \ref{Cohen-Macaulay-stb}.

The remaining equivalences are all of them a consequence of the fact that $M_y(F)$ is flat over $\mathcal O_r$ if and only if $\cO_{r+n+1}/J_y(G)$ is flat over $\mathcal O_r$. In fact, 
we consider the exact sequence defining $M_y(F)$:
$$
0\longrightarrow M_y(F)\longrightarrow\frac{\mathscr C (F)}{J_y(G)}\longrightarrow\frac{C (F)}{J_y(G)\cdot \cO_{r+n}}\longrightarrow 0.$$
Since the last module is $\cO_r$-flat by Lemma \ref{lem:jacob-source} and Remark \ref{rem:ramif} the first module is $\cO_r$-flat if and only if the  second is. Now consider the exact sequence:
$$0\longrightarrow \frac{\mathscr C (F)}{J_y(G)}\longrightarrow\frac{\cO_{r+n+1}}{J_y(G)}\longrightarrow\frac{\cO_{r+n+1}}{\mathscr C (F)}\longrightarrow 0.$$
The last module of the sequence is $\cO_r$-flat by Remark \ref{rem:determinantal}, and hence the 
$\cO_r$-flatness is equivalent for the first two modules of the sequence.
\end{proof}

\begin{defn}
We say that a germ $f$ \emph{has the Cohen-Macaulay property} if $\dim_\C M(f)<\infty$ and $M_y(F)$ is Cohen-Macaulay of dimension $r$
for some stable unfolding $F$.
\end{defn}

Here is the relation with Mond's conjecture:

\begin{thm}\label{main2}
Let $f\colon (\C^n,S)\to(\C^{n+1},0)$ be $\mathscr A$-finite, with $(n,n+1)$ nice dimensions and $n\ge 2$. If $f$ has the Cohen-Macaulay property then $f$ satisfies Mond's conjecture. Moreover, if $f$ is 
weighted homogeneous, then the converse is true.
\end{thm}
\begin{proof}
Suppose that $f$ has the Cohen-Macaulay property. In Corollary~\ref{Cohen-Macaulay-stb} we proved the equality
$$
\dim_\C M(f)=\alpha(f)+\mu_I(f),
$$
but since we are in the nice dimensions, $\alpha(f)=0$. Mond's conjecture for $f$ follows now inmediately from Corollary \ref{M(f)=0}.

Suppose that $f$ is a weighted homogeneous germ satisfying Mond's conjecture, Then $\dim_\C M(f)=\mu_I(f)$ by Corollary \ref{M(f)=0}. Since $\alpha(f)=0$, Corollary~\ref{Cohen-Macaulay-stb} implies that $f$ has the Cohen-Macaulay property.
\end{proof}

\section{Reduction of Mond's conjecture to families of examples}
We exploit the results in previous section to reduce the general validity of Mond's conjecture for map germs to its validity in suitable families of examples.

We call multiplicity of a map germ $f\colon(\C^n,S)\to (\C^{n+1},0)$ to the minimum of the multiplicities of the components $(f_1,...,f_{n+1})$ of $f$ at all the points in $S$.

\begin{thm}\label{reduction1}
Assume that we are in the nice dimensions.
Suppose that for any natural number $M$ there exists a weighted homogenous $\mathscr A$-finite germ $h\colon (\C^n,S)\to (\C^{n+1},0)$
of multiplicity at least $M$ for which Mond's conjecture holds. Then Mond's conjecture holds for any $\mathscr A$-finite germ 
$f\colon(\C^n,S)\to (\C^{n+1},0)$.
\end{thm}
\begin{proof}
Let $f\colon(\C^n,S)\to (\C^{n+1},0)$ be $\mathscr A$-finite. By finite determinacy we may assume, up to right-left equivalence, that $f$ is $M$-determined for a certain natural number $M$. 
Let $h$ be the germ predicted by hypothesis.
By $M$-determinacy, if we consider the $1$-parameter family of germs
$h_t:=h+tf\colon(\C^n,S)\to (\C^{n+1},0)$ we have that $h_t$ is equivalent to $f$ for any $t\neq 0$. 

Let $H$ be a stable and versal unfolding of $h$ parametrised by a base $T$. Since Mond's conjecture holds for $h$ and $h$ is weighted homogeneous, by Theorem~\ref{main2}, the module $M_y(H)$ is $T$-flat. By versality, and because 
$h_t$ is equivalent to $f$ for any $t\neq 0$ we have that $H$ is also a versal unfolding of $f$. Thus, applying again Theorem~\ref{main2}, we obtain Mond's conjecture for $f$.
\end{proof}

A similar result can be proved if we want to study maps of a certain corank. For simplicity, we state the result only for mono-germs, although it can be easily adapted to multi-germs, if necessary. An $\mathscr A$-finite $f\colon(\C^n,0)\to (\C^{n+1},0)$ of corank $r$ is right-left equivalent to a map germ whose components admit the normal form
$$(x_1,...,x_{n-r},f_{n-r+1},...,f_{n+1}).$$ We call the {\em corank-$r$ multiplicity} of $f$ the minimum of the multiplicities of $f_{n-r+1},...,f_{n+1}$. 

\begin{thm}\label{reduction2}
Assume that we are in the nice dimensions.
Suppose that for any natural number $M$ there exists a weighted homogenous $\mathscr A$-finite germ $h\colon(\C^n,0)\to (\C^{n+1},0)$ of corank $r$ and 
of corank-$r$ multiplicity at least $M$ for which Mond's conjecture holds. Then Mond's conjecture holds for any $\mathscr A$-finite germ of corank $r$.
$f\colon(\C^n,0)\to (\C^{n+1},0)$.
\end{thm}
\begin{proof}
The proof is entirely analogous to the proof of the previous theorem. The only modification is that one needs to put $f$ and $h$ in normal form before constructing the family $h_t$.
\end{proof}


\begin{thebibliography}{99}




\bibitem{AGV} V.I. Arnold, S.M. Gusein-Zade and A.N. Varchenko.
Singularities of differentiable maps, Volume 2. Monodromy and asymptotics of integrals. Reprint of the 1988 hardback edition.
Modern Birkhäuser Classics. Boston, MA: Birkhäuser

\bibitem{BM} J.W. Bruce and W.L. Marar, Images and varieties. {\it J. Math. Sciences}, {\bf 82} (5), 3633--3641.

%

\bibitem{DM} J. Damon and D. Mond, 
$\mathscr A$-codimension and the vanishing topology of discriminants.
{\it Invent. Math.}, {\bf 106} (1991), no. 2, 217--242. 

\bibitem{MG} T. Gaffney, Properties of finitely determined germs, Phd thesis, Brandeis University, 1976.

\bibitem{singular}G.M. Greuel, G. Pfister, {\it A Singular
introduction to commutative algebra.} Springer, Berlin, 2008.

\bibitem{Hert} C. Hertling.
Frobenius manifolds and moduli spaces for singularities. 
Cambridge Tracts in Mathematics. 151. Cambridge University Press., 270 p, (2002).

\bibitem{Kulk} V. Kulikov. Mixed Hodge structures and singularities. 
Cambridge Tracts in Mathematics. 132. Cambridge University Press., 186 p., (1998).

%
%

\bibitem{JS} T. de Jong and D. van Straten, Disentanglements, Singularity theory and its applications, Part I (Coventry, 1988/1989), 93--121, Lecture Notes in Math., 1462, Springer, Berlin, 1991. 

\bibitem{MM} W.L. Marar and D. Mond, Multiple point schemes for corank $1$
maps, {\it J. London Math. Soc.}, {\bf 39} (1989), 553--567.

\bibitem{MatherII}  J.N. Mather, Stability of $C^\infty$ mappings. II. Infinitesimal stability implies stability. {\it Ann. of Math.} (2) {\bf 89} (1969), 254--291.

\bibitem{MatherIII} J.N. Mather, Stability of $C^\infty$ mappings. III. Finitely determined map-germs. {\it Inst. Hautes Études Sci. Publ. Math.} No. {\bf 35} 1968 279--308.

\bibitem{MatherVI} J.N. Mather, Stability of $C^\infty$ mappings. VI: The nice dimensions. Proceedings of Liverpool Singularities-Symposium, I (1969/70), pp. 207--253. Lecture Notes in Math., Vol. 192, Springer, Berlin, 1971.

\bibitem{Mond:vanishing} D. Mond, 
Vanishing cycles for analytic maps. Singularity theory and its applications, Part I (Coventry, 1988/1989), 221--234, Lecture Notes in Math., 1462, Springer, Berlin, 1991. 

\bibitem{Mond:bent wires} D. Mond, Looking at bent wires---$\mathscr A_e$-codimensions and the vanishing topology of parametrised curve singularities, {\it Math. Proc. Cambridge Phil. Soc.} {\bf 117}  no. 2 (1995), 213--222.

\bibitem{Mond: Sao Carlos} D. Mond, Lectures on singularities of mappings from $n$-space to $n+1$-space, Real and Complex Singularities São Carlos 2013, 1--89, Adv. Studies Pure Math., Math. Soc. Japan, 2015.

\bibitem{Mond:Edinburgh} D. Mond, Some open problems in the theory of singularities of mappings, {\it J. Sing.} (2015).

\bibitem{MP} D. Mond, R. Pellikaan, Fitting ideals and multiple points of
analytic mappings. Algebraic geometry and complex analysis (P\'atzcuaro, 1987), 107--161, Lecture
Notes in Math., 1414, Springer, Berlin, 1989.

\bibitem{P} R. Pellikaan, Hypersurface singularities and resolutions of Jacobi modules. Thesis, Rijksuniversiteit Utrech, 1985.

\bibitem{RP} R. Piene, Ideals associated to a desingularitation. Proc, Summer Meeting Copenhagen 1978, 503--517, Lecture Notes in Math., 732, Springer, Berlin, 1979.

\bibitem{Siersma} D. Siersma, 
Vanishing cycles and special fibres. Singularity theory and its applications, Part I (Coventry, 1988/1989), 292--301, Lecture Notes in Math., 1462, Springer, Berlin, 1991.

\bibitem{Siersma3} D. Siersma. 
The vanishing topology of non isolated singularities. Siersma, D. (ed.) et al., New developments in singularity theory. Proceedings of the NATO Advanced Study Institute on new developments in singularity theory, 
Cambridge, UK, July 31-August 11, 2000. Dordrecht: Kluwer Academic Publishers. NATO Sci. Ser. II, Math. Phys. Chem. 21, 447-472 (2001).
\bibitem{W} C.T.C. Wall, Finite determinacy of smooth
map germs, {\it Bull. London Math. Soc.}, {\bf 13} (1981), 481--539.


\end{thebibliography}
\end{document}